\documentclass{amsart}
\usepackage{amssymb,amsmath,amsrefs}
\usepackage[colorlinks=true]{hyperref}
\hypersetup{urlcolor=blue, citecolor=red}
\usepackage[dvips]{graphicx}
\usepackage{color}

\usepackage{tabularx}
\theoremstyle{plain}

\newtheorem{mainthm}{Theorem}

\newtheorem*{conj*}{Conjecture}
\newtheorem*{cor*}{Corollary}
\newtheorem{theorem}{Theorem}[section]
\newtheorem{thm}[theorem]{Theorem}

\newtheorem{proposition}[theorem]{Proposition}
\newtheorem{corollary}[theorem]{Corollary}
\newtheorem{lemma}[theorem]{Lemma}

\newtheorem{claim}{Claim}

\theoremstyle{definition}
\newtheorem*{def*}{Definition}
\newtheorem{remark}[theorem]{Remark}
\newtheorem{rmk}[theorem]{Remark}

\newtheorem{definition}[theorem]{Definition}

\renewcommand{\SS}{{\mathcal S}}

\newcommand{\SU}{{\mathcal U}}

\newcommand{\de} {\delta}

\renewcommand{\epsilon}{\varepsilon}

\newcommand{\Z}{\mathbb{Z}}

\newcommand{\R}{\mathbb{R}}
\newcommand{\eps}{\varepsilon}

\newcommand{\dist}{\operatorname{\textit{d}}}

\newcommand{\diam}{\operatorname{diam}}		
\DeclareMathOperator{\orb}{orb}
\makeatletter
\newcommand{\tpitchfork}{%
  \vbox{
    \baselineskip\z@skip
    \lineskip-.52ex
    \lineskiplimit\maxdimen
    \m@th
    \ialign{##\crcr\hidewidth\smash{$-$}\hidewidth\crcr$\pitchfork$\crcr}
  }%
}
\makeatother

\title{Minimal Expansive Flows}

\author[A. Artigue]{Alfonso Artigue}
\address[A. Artigue]{Universidad de la República, CENUR Litoral Norte, DMEL, Paysandú, Uruguay.}
\email{aartigue@litoralnorte.udelar.edu.uy}

\author[E.Rego]{Elias Rego}
\address[E. Rego]{AGH University of Krakow, Faculty of Applied Mathematics, Krak\'ow, Poland.}
\email{rego@agh.edu.pl}



\begin{document}

\title{Expansive minimal flows}

\begin{abstract}
We extend a Ma\~n\'e's famous result on expansive homeomorphisms, originally presented in \cite{Ma}, to the setting of flows. Specifically, we provide a complete characterization of minimal expansive flows on compact metric spaces. We prove that such flows must be defined on one-dimensional sets and are equivalent to the suspension of a minimal subshift. This result significantly improves upon \cite{KS} by eliminating the need for their additional hypothesis. Furthermore, we apply our findings to show that any regular expansive flow on a compact metric space of dimension two or higher must contain infinitely many minimal subsets.   
\end{abstract}
\maketitle

\section{Introduction}
The concept of expansiveness is a cornerstone of dynamical systems theory. Firstly introduced in \cite{U}, it was later shown that this property naturally arises in hyperbolic systems. In fact, the relationship between expansiveness and hyperbolicity runs deep: together with the shadowing property, expansiveness forms the topological foundation of hyperbolic systems. Beyond its connection to hyperbolic theory, expansive systems are of intrinsic interest in their own right. They are defined independently of any differentiability assumptions and give rise to a robust theory with numerous significant implications. For example, expansiveness is a key tool in the study of stability phenomena. Moreover, expansive systems are frequently regarded as a rich source of chaotic behavior.

Since its inception, the theory of expansive homeomorphisms has undergone significant development, evolving into a rich framework that provides powerful tools for analyzing the dynamics of discrete-time systems. This theory has been widely generalized to various forms of expansivity, both from a topological and measurable perspective. For comprehensive expositions on expansive homeomorphisms and their generalizations, the reader is referred to \cite{AH} and \cite{MoSi}.

Beyond the realm of homeomorphisms, the concept of expansiveness was extended to flows in \cite{BW}. Initially aimed at capturing the expansive behavior of Axiom A flows, expansiveness has proven to be a crucial tool for understanding the chaotic dynamics of regular flows. Due to its broad applicability, expansive flows have been extensively used to study hyperbolic-like systems, as well as other types of flows, such as geodesic flows, which naturally emerge in the study of Riemannian manifolds (see, for example, \cite{CP} and \cite{Pa}).

However, despite the extensive theory available for expansive homeomorphisms, there remains a noticeable gap in results for expansive flows. Although some extensions have been made, many fundamental results available for expansive discrete-time systems are still lacking in the context of flows. There are two main distinctions between the homeomorphism and the flow setting:

\begin{enumerate}
    \item On flows, expansiveness cannot be perceived within small orbit arcs.
    \item We need to consider time changes when dealing with flows. 
 \end{enumerate}   
    
    These differences impose certain limitations when attempting to apply homeomorphism techniques to the flow setting, making the extension of many results particularly challenging. In this work, we aim to extend a result by Ma\~n\'e in \cite{Ma}, which relates minimality and expansiveness for homeomorphisms, to the context of flows.
    
    In what follows, we precisely state our main results. Throughout this work $(X,\dist)$ denotes a compact metric space.
\begin{definition}
    A continuous flow on $X$ is a continuous map  $\phi\colon \R\times X\to X$ satisfying 
    \begin{enumerate}
        \item $\phi(0,x)=x$, for every $x\in X$. 
        \item $\phi(t+s,x)=\phi(s,\phi(t,s))$, for every $x\in X$ and $t,s\in \R$.
        \end{enumerate}
\end{definition}
 We denote by $\phi_t$ the map $\phi(t,\cdot)$. A point $x\in X$ is a \textit{fixed} point for $\phi$ if $\phi_t(x)=x$, for every $t\in \R$. We say that $\phi$ is \textit{regular} if there are not fixed point for $\phi$. A compact subset $K\subset M$ is said to be $\phi$-\textit{invariant} if $\phi_t(K)=K$, for every $t\in \R$. A compact and invariant subset is \textit{minimal} if its only proper, compact and invariant subset is the empty set. 
\begin{definition}
A flow $\phi$ on $X$ is said to be minimal if $X$ is a minimal set.    
\end{definition}
Observe that the trivial flow, i.e., when $X$ is a singleton, is trivially minimal. Also, every non-trivial minimal flow is regular.  Next, we introduce the main subject of this work. But first, let us denote by $Rep$ the set of increasing homeomorphisms $h:\R\to\R$ such that $h(0)=0$.
 
\begin{definition}\label{Def: exp}
    A flow $\phi$ is expansive if for every $\eps>0$, there is $\delta>0$ such that if $x,y\in X$ and  $h\in Rep$  satisfy 
    $$d(X_t(x),X_{h(t)}(y))\leq \delta,$$
    for every $t\in \R$, then $y\in \phi_{[-\eps,\eps]}(x)$. 
\end{definition}

\begin{rmk}
    Expansiveness for flows was initially introduced in \cite{BW}, but not in the previous form. Nevertheless, it is shown in the same paper that the previous definition and the original one coincide.
\end{rmk}

The relation between minimality and expansivity was originally explored in \cite{Ma}, where Ma\~n\'e proved that, for homeomorphisms, these properties are not compatible, unless $X$ is a totally disconnected set. Precisely:

\begin{theorem}[\cite{Ma}]
    Let $f$ be an expansive homeomorphism of a compact metric space $X$. If $f$ is minimal, then $X$ has null topological dimension. 
\end{theorem}

The only known attempt to transpose Ma\~n\'e's result to the setting of flows was made in \cite{KS}. In their work, Keynes and Sears provided a partial extension of Ma\~n\'e's result to expansive flows, but with an additional assumption: the non-existence of spiral points (see Section \ref{Sec2} for the precise definition).

\begin{theorem}[\cite{KS}]\label{thmKS}
    Let $\phi$ be an expansive flow on a compact metric space $X$. If $\phi$ is minimal and has no spiral points, then $X$ has topological dimension one. 
\end{theorem}

The issue of spiral points arises exclusively in the context of expansive flows due to the need for considering reparametrizations. Overcoming this obstacle is non-trivial, but in \cite{KS}, the authors managed to eliminate the spiral points condition for Axiom A flows, a much stronger setting than expansiveness (within the non-wandering set). This naturally led us to the following question:

\vspace{0.1in}
\textit{\textbf{Question:} Is it possible to obtain Theorem \ref{thmKS} without assuming the no spiral points condition?}

\vspace{0.1in}

This question has remained open since 1981. In our first main result, we provide a positive answer to it. 

\begin{mainthm}\label{minexp}
    If $\phi$ is a non-trivial minimal expansive flow on $X$, then the topological dimension of $X$ is equal to one. Consequently, a flow is minimal and expansive if, and only if, it is a suspension of a minimal subshift. 
\end{mainthm}

The main idea behind proving Theorem \ref{minexp} is to use more modern techniques to demonstrate that the hypothesis in the original result from \cite{KS} is superfluous. We achieve this by applying the theory of cross-sections introduced in \cite{ArFCS} in conjunction with the catenary sectional metric from \cite{ArCatenary}. Specifically, we show that the existence of spiral points would lead to the presence of periodic points for
$\phi$. Consequently, this would imply either $X$ is a periodic orbit or $\phi$ is not minimal.

  With these ideas in mind, we continue to explore the role of minimal sets in expansive dynamics. Our techniques also extend to studying other features of expansive flows. Another long-standing open question is whether expansive flows, aside from the minimal case, admit periodic orbits. Note that periodic orbits are indeed minimal sets. Additionally, it is classical that every flow contains at least one minimal subset. Using our techniques, we show that, although the existence of periodic orbits in expansive flows remains unresolved, we can still establish the presence of infinitely many minimal sets.

\begin{mainthm}\label{minimalsets}
    Suppose the topological dimension of $X$ is bigger than one.  If $\phi$ is an expansive flow on $X$, then $X$ contains infinitely many minimal subsets for $\phi$.
\end{mainthm}

In \cite{Wi} there is another result for homeomorphisms whose proof depends on finding periodic orbits in the $\omega$-limit set of a spiral point. Precisely,  it is shown that if $X$ is compact and $f$ is a homeomorphism of $X$ which is expansive on 
$X\setminus\cup_{i=1}^N\orb(x_i)$ for some $N\geq 1$ and $x_1,\dots,x_N\in X$, then $f$ is expansive on $X$.
In the translation of this result to the flow scenario we cannot apply our previous results because we need expansivity on a compact set. However, with different techniques we can deal with the spiral points and show the next result.

\begin{mainthm}\label{charac}
A regular flow $\phi$ which is expansive on $X\setminus\orb(x_1)$ for $x_1\in X$ is expansive on $X$.
\end{mainthm}

In \cite{ArExpFSurf} singular expansive flows of surfaces are classified. It is essentially shown that they are conjugate to singular suspensions of minimal interval exchange maps. These kinds of flows have fixed points and are expansive in the sense of Komuro: 
The flow $\phi$ is $k^*$-\textit{expansive} if for all $\epsilon>0$ there is $\delta>0$ such that 
if $x,y\in X$, $h\in Rep$ and $d(\phi_{h(t)}(x),\phi_t(y))<\delta$ for all $t\in\R$ then 
$\phi_{h(t_0)}(x)\in\phi_{[t_0-\epsilon,t_0+\epsilon]}(y)$ for some $t_0\in\R$.
This form of expansivity was introduced by Komuro to study the Lorenz attractor and is shown to be equivalent to Definition \ref{Def: exp}, if the flow is regular.
For the results in \cite{ArExpFSurf} it is important to show that $k^*$-expansivity is invariant under adding or removing index zero singular points. This was shown for surface flows. 
We will conclude from Theorem \ref{charac} that a regular flow is expansive if and only if it is $k^*$-expansive after adding an \textit{index zero} singularity.

The remainder of this paper is organized as follows. In Section \ref{Sec2}, we introduce the foundational background and present some relevant results about regular flows. Section \ref{SecSpiral} is dedicated to studying the behavior of spiral points in expansive flows and provides a proof for Theorem \ref{minexp}. In Section \ref{SecMinSets}, we prove Theorem \ref{minimalsets}.
In Section \ref{secCharac} we prove Theorem \ref{charac}.
\section{Preliminaries}\label{Sec2}

In this section, we present the main definitions used in this work and provide the reader with some previously established results on the theory of expansive flows that will be instrumental in our analysis. From here on, let $(X,d)$ denote a compact metric space.

\begin{definition}[Topological Dimension]
 We say that  $X$ has topological dimension $n$ and denote $\dim(X)=n$ if every open cover $\mathcal{U}$ of $X$ admits a subcover $\mathcal{U}'$ such that any point of $X$ is contained in at most $n+1$ elements of $\mathcal{U}'$. If such a $n$ does not exist, we say that $\dim(M)=\infty$. 
\end{definition}

\begin{rmk}
    We would like to remark the following basic facts about topological dimension.
    \begin{enumerate}
			\item The subcover in the previous definitions can always be chosen to be finite, due the compactness of $X$.
    \item The topological dimension of $X$ is invariant under homeomorphisms.
    \item If $X=[0,1]^n$, then $\dim(X)=n$.
    \end{enumerate}
\end{rmk}

Hereafter,  $\phi$ denotes a continuous flow over $X$. All the flows under consideration in this work will always be regular, unless otherwise stated.  The \textit{orbit} of a point $x$ under $\phi$ is the set $$\orb(x)=\{\phi_t(x), t\in \R\}.$$ 
Also, the positive and negative orbits of $x$ are respectively the sets
 $$ \orb^+(x)=\{\phi_t(x), t\geq 0\} \textrm{ and } \orb^-(x)=\{\phi_t(x), t\leq 0\}.    $$
For any $x\in X$, define the {\it $\omega$-limit set} of $x$ by
$$\omega(x)=\{z\in X; \textrm{ there is } t_k\to \infty, \phi_{t_k}(x)\to z\}.$$
Recall that there sets $\overline{\orb(x)}$ and $\omega(x)$ are compact and $\phi$-invariant subsets of $X$ and $\omega(x)\subset \overline{\orb^+(x)}$.  A classical fact is that every compact and invariant subset always contains a non-trivial minimal subset. In particular, $\omega(x)$ always contains a non-trivial minimal subset, for every $x\in X$. 

\begin{rmk}
    It is well known that $\phi$ is minimal if, and only if, $\overline{\orb(x)}=X$, for every $x\in X$.
\end{rmk}
 
Let $\mathcal{K}(X)$ be the set of compact subsets of $X$. One can make $\mathcal{K}(X)$ a compact metric space by endowing it with the Hausdorff metric:
$$d_H(A,B)=\inf\{\eps>0; A\subset B_{\eps}(B) \textrm{ and }B\subset B_{\eps}(A)\}.  $$

\begin{definition}
	Let $x\in X$ and $\tau>0$. A set  $A\in \mathcal{K}(X)$ is a \textit{cross-section} of time $\tau$ through $x$,
if $x\in A$, 
$\phi_{[-\tau,\tau]}(A)$ is a neighborhood of $x$ and for every $y\in A$, one has $\phi_{[-\tau,\tau]}(y)\cap A=\{y\}$.
\end{definition}

Next, we recall the concept of field of cross-sections introduced in \cite{ArFCS}.

\begin{definition}
A map $H\colon X \to \mathcal{K}(X)$ is a field of cross-sections for $\phi$ if there is $\tau>0$ such that $H(x)$ is a cross-section through $x$ and time $\tau$, for every $x\in X$.
\end{definition}

In \cite{ArFCS}*{Theorem 2.51} it is proved that every regular flow admits a semicontinuous field of cross-sections.
If $H$ is such a field of cross-sections and  $\eps>0$  let us denote
$H_\eps(x)=H(x)\cap \overline{B_\epsilon(x)}$. 
In this way, we obtain a field of cross-sections with arbitrarily small diameter and time. 
 Precisely, if $\phi$ is a regular flow, there are $\delta_0,\tau_0>0$   such that for any $0<\delta<\delta_0$ and $0<\tau<\tau_0$ then $H_{\delta}(x)$ is a semicontinuous field of cross-sections 
of time $\tau$ and diameter less than $2\delta$
 for every $x\in X$.

Another important tool is the concept of partial flow. To define it, suppose $H=H_\eps$ is a continuous field of cross-sections and denote  \[N_\epsilon=\{(x,y)\in X\times X:y\in H_\epsilon(x)\}.
\]
Following \cite{ArCatenary}*{Definition 2.1}, we can define a partial flow on $N_\epsilon$. 	A partial flow is essentially a flow whose orbits may not be defined for all $t\in\R$, just as in the basic theory of differential equations.
In our case, for $(x,y)\in N_\epsilon$ and $t\in\R$ (suppose $t>0$)   we define $\psi_t(x,y)$ ($\psi$ is the partial flow) only if there is $h\colon [0,t]\to\R$ continuous such that $h(0)=0$ and $\phi_{h(s)}(y)\in H_\epsilon(\phi_s(x))$ for all $s\in [0,t]$. In this case, we define
$\psi_t(x,y)=(\phi_t(x),\phi_{h(t)}(y))$. By \cite{ArFCS}*{Lemma 3.1}, this reparametrization $h$ is unique and the partial flow is well defined.

\begin{remark}\label{rmk1}
Observe that, by the continuity of $\phi$, if $(x,y)\in N_{\eps}$ and  $y$ is close to $x$, there is $t>0$ so that the partial flow $\psi_s(x,y)$ is well defined for every $0\leq s\leq t$. On the other hand, as long as $t$ grows, $\psi_t(x,y)$ may not be defined. Nevertheless, by classical arguments (see for instance \cite{ArFCS}) 
there is $\delta>0$ such that
if $h\in Rep$ and $d(\phi_{t}(x),\phi_{h(t)}(y))<\delta$ for every $t\geq 0$, then by possibly modifying $h$, we can assume $\phi_{h(t)}(y)\in H_{\eps}(x)$, for every $t\geq 0$. Consequently, then $\psi_t(x,y)$ is defined for every $t\geq 0$. 
\end{remark}

Next, we define a crucial ingredient of our proof. A \textit{sectional metric} (see \cite{ArCatenary}),  is a continuous map
\[
D\colon \{(x,y,z)\in X^3:y,z\in H_\epsilon(x)\}\to\R
\]
such that $D_x=D(x,\cdot,\cdot)$ is a metric on $H_\epsilon(x)$ for all $x\in X$. 
Suppose that $y,z\in H_\epsilon(x)$, $\delta>0$ and there are continuous functions
	$h_y,h_z\colon (-\delta,\delta)\to\R$ such that
\begin{itemize}
	\item $h_y(0)=0$, $h_z(0)=0$ 
	\item for all $|t|<\delta$
	\[
	\left\{
	\begin{array}{l}
		\phi_{h_y(t)}(y)\in H_\epsilon(\phi_t(x)),\\
		\phi_{h_z(t)}(z)\in H_\epsilon(\phi_t(x)).
	\end{array}
	\right.
	\]
\end{itemize}
The time derivatives are indicated by dots and defined as
\[
\dot D_x(y,z)=\frac{d}{dt} D_{\phi_t(x)}(\phi_{h_y(t)}(y),\phi_{h_z(t)}(z))|_{t=0}
\]
and
\[
\ddot D_x(y,z)=\frac{d}{dt} \dot D_{\phi_t(x)}(\phi_{h_y(t)}(y),\phi_{h_z(t)}(z))|_{t=0}.
\]

\begin{definition}
	A sectional metric $D$ is \textit{catenary} if it satisfies $\ddot D_x=D_x$ for all $x\in X$.
\end{definition}
The next result asserts the existence of catenary sectional metrics for any expansive flow.
\begin{theorem}[\cite{ArCatenary}*{Theorem 5.7}]
	If $\phi$ is an expansive flow and $H$ is a field of cross-section for $\phi$, then $H$ admits a catenary sectional metric.
\end{theorem}

\begin{rmk}\label{rmk2}
By solving the second-order differential equation $\ddot D=D$ we have that
$$D_{\phi_t(x)}(\phi_{h_y(t)}(y),\phi_{h_z(t)}(z))=a e^{-t}+be^t$$
for all $t$ where the partial flow is defined. Moreover, the definition of the partial flow depends on the triple $(x,y,z)$.  In particular, if there are $(x,y,z)$  so that the partial flow is defined for every $t\geq 0$, considering that the space is compact and the sectional metric is bounded, we have $b=0$ and 
\begin{equation}
D_{\phi_t(x)}(\phi_{h_y(t)}(y),\phi_{h_z(t)}(z))=D_x(y,z) e^{-t}
\end{equation}
for all $t\geq 0$.  
\end{rmk}

Finally, we introduce one of the main subjects in this work: the spiral points. 

\begin{definition}[\cites{ArMin,KS}] 
	We say that $x\in X$ is a \emph{spiral point} if there are $\tau>0$ and a continuous function $g\colon\R\to\R$ satisfying
$g(t)>\tau+t$ for all $t\geq 0$ and 
\[
\dist(\phi_t(x),\phi_{g(t)}(x))\to 0\text{ as }t\to+\infty.
\]
\end{definition}

\begin{remark}\label{rmk3}
Suppose $\phi$ is an expansive flow, $H_\eps$ is a field of cross-section with a catenary metric $D$. Observe that, if $x$ is a spiral point, then for every $0<\delta\leq \eps$,  up to replacing $x$ by a future iterate we can assume $\dist(\phi_t(x),\phi_{g(t)}(x))\leq \delta$, for every $t\geq 0$. Therefore, by combining Remarks \ref{rmk1} and \ref{rmk2}, we obtain:
\begin{equation}
\label{ecuCatenaria}
D_{\phi_t(x)}(\phi_{t}(x),\phi_{g(t)}(x))=D_x(x,\phi_{g(0)}(x)) e^{-t}.
\end{equation}
\end{remark}

\section{Spiral Points and Periodic Orbits}\label{SecSpiral}

This section is dedicated to proving Theorem \ref{minexp}. The central idea in our proof is to explore the relationship between periodic orbits and spiral points for expansive flows. Similarly to the discrete-time case, we will obtain that, under the expansive assumption, periodic orbits must appear in the $\omega$-limit set of any spiral point, as stated in the following key result.

\begin{thm}\label{spiral}
If $\phi$ is expansive and $x_0$ is a spiral point, then $\omega(x_0)$ is a periodic orbit.
\end{thm}

The proof of Theorem \ref{minexp} can be easily derived from the previous result. 

\begin{proof}[Proof of Theorem \ref{minexp}]
Let $\phi$ be an expansive minimal flow on  $X$ and suppose $x_0$ is a spiral point. If $X$ is a periodic orbit, then there is nothing to show. Now, suppose $X$ is not a periodic orbit.  By Theorem \ref{spiral},  $\omega(x_0)$ is a periodic orbit. By the minimality of $\phi$, then $X$ is the sole periodic orbit, which is a contradiction. Thus, $X$ cannot contain a spiral point; therefore, \cite{KS}*{Theorem 3.6} implies $\dim(X)=1$, and the proof is complete.

To prove the second part of the theorem, we first notice that by \cite{BW}*{Theorem 6}, the suspension flow of any minimal expansive subshift is  
 minimal and expansive. Conversely, if a flow is minimal and expansive, then $\dim(M)=1$. Thus by \cite{BW}*{Theorem 10}, $\phi$ is the suspension of a subshift which must be minimal and expansive, \cite{BW}*{Theorem 6}.
\end{proof}

The remainder of this section is devoted to the proof of Theorem \ref{spiral}. The main obstacle, in contrast to the discrete-time case, is that the reparametrizations involved in the definition of a spiral point may lead to unbounded return times. This poses a significant difficulty in locating periodic points within the omega-limit set. Nevertheless, the hyperbolic estimates provided by the catenary metrics will ensure the boundedness of return times. 

Hereafter, we suppose $\phi$ is a regular flow on $X$ and fix $\de_0>0$ and $\tau_0>0$ such that, for every $0<\delta<\de_0$ and $0<\eps<\tau_0$, there is a field of cross-sections $H$ with time $\eps$ such that $H(x)= H_{\delta}(x)$, for every $x\in X$. Before beginning with the proof, we need the following lemma:

\begin{lemma}[\cite{ArFCS}*{Lemma 3.2}]
	\label{lemConvUnifParam}
	Let $x,y\in X$ and $a>0$. If $h_n\colon [0,a]\to\R$ are increasing continuous functions satisfying 
	$h_n(0)=0$,
	$x_n\to x$, $y_n\to y$ satisfy
	$\phi_{h_n(t)}(y_n)\in H(\phi_t(x_n))$ for all $n\geq 1$ and for all $0\leq t\leq a$ then $h_n$  converges uniformly to some $h\colon [0,a]\to\R$ satisfying $\phi_{h(t)}(y)\in H(\phi_t(x))$ for all $0\leq t\leq a$.
\end{lemma}

Now we are able to prove Theorem \ref{spiral}.

\begin{proof}[Proof of Theorem \ref{spiral}]
Suppose $\phi$ is an expansive flow and let $x_0$ be a spiral point of $x_0$. First, we will show that $\omega(x_0)$ contains a periodic orbit.
Let $0<\eps<\tau_0$ and let $0<\delta<\delta_0$ given by the expansiveness of $\phi$ and fix a field of cross-section for $\phi$ with time $\eps$ and such that $H_{\de}(x)= H(x)$, for every $x\in X$. Fix a constant $0<\gamma<\epsilon$.
By the definition of cross-section, 
	$$\phi_{[-\eps,\eps]}(x)\cap H_\delta(x_0)=\{x\}.$$ for all $x\in H_\delta(x_0)$.
Thus, we can consider the flow box:
$$F(x_0,\delta,\eps)=\phi_{[-\eps,\eps]}(H_\delta(x_0)).$$ Hence, if $y\in F(x_0,\delta,\eps)$, then there are a unique $z_y\in H_{\eps}(x_0)$ and a unique time $s_y$ such that  $|s_y|\leq\eps$ and $\phi_{s_y}(z_y)=y$. By the previous remarks,  the flow induces a canonical projection
$$\pi\colon F(x_0,\delta,\eps)\to H_\delta(x_0),$$
by setting $\pi(y)=z_y$. Observe that the projection is clearly continuous, due to the continuity of $\phi$.
Consequently, there is $0<\eta<\gamma/3$ such that if 
for $$y\in H_\delta(x)
\textrm{ and } z\in H_\delta(y)\cap F(x,\delta,\eps)$$ we have
\begin{equation}
\label{ecuCambioSec}
\text{if }D_x(x,y)<\eta\text{ then }|D_y(y,z)-D_x(y,\pi(z))|<\gamma/3.
\end{equation}

Observe that, by possibly taking a future iterate, we can assume that the spiral point $x_0$ satisfies:
\begin{equation}
\label{ecuDeltaEspiral}
\phi_{g(t)}(x_0)\in H_\delta(\phi_t(x_0)),\textrm{  }
 D_{\phi_t(x_0)}(\phi_{g(t)}(x_0),\phi_t(x_0))<\eta,
 \end{equation}
and 
$$g(t)-t>\tau,$$ for every $t\geq 0$. Now, we split the proof into two cases:

\textbf{Case 1:} Suppose   $$\liminf\limits_{t\to+\infty}g(t)-t<\infty.$$ In this case, we can easily obtain a periodic orbit in $\omega(x_0)$ by the continuity of $\phi$. Indeed,  there are a sequence $t_n\to \infty$ and $B>0$ such that $\tau<g(t_n)-t_n<B  $, for every $n\geq 0$. Consequently, are up to take a subsequence, we obtain $g(t_n)-t_n\to s>0$ and $\phi_{t_n}(x_0)\to p\in \omega(x)$. Hence, $\phi_{g(t_n)-t_n}(\phi_{t_n}(x_0))\to \phi_{s}(p)$. But the definition of spiral point implies
$$d(\phi_{t_{n}}(x_0),\phi_{g_{t_n}}(x_0))\to 0,$$
and therefore $\phi_s(p)=p$.  \\

\textbf{Case 2:} Suppose  $$\liminf\limits_{t\to+\infty}g(t)-t=\infty.$$ This case is more delicate, and we will have to use the catenary metric to obtain a periodic orbit in $\omega(x_0)$. Again, by possibly taking a future iterate of $x_0$, we can assume $e^{-\tau}<\frac{1}{3}$.  \\

Next, we shall define a sequence of points $(x_n)_{n\geq 0}$  such that, for every $n\geq 0$:
\begin{itemize}
    \item $x_n\in H_\delta(x_0)\cap \phi_{[0,+\infty)}(x_0)$
    \item $D_{x_0}(x_0,x_n)<\gamma$ for all $n\geq 0$.
\end{itemize}  
Start with $x_0$ and $x_1=\phi_{g(0)}(x_0)$.
Next, define $$x'_2=\phi_{g(g(0))-g(0)}(x_1)=
\phi_{g(g(0))}(x_0).$$
Notice that $x_1\in H_{\delta}(x_0)$ and $x'_2\in H_{\delta}(x_1)$, but it does not necessarily implies $x'_2\in H_{\delta}(x_1)$. 
On the other hand, by applying \eqref{ecuCatenaria} and \eqref{ecuDeltaEspiral} we have
\[
D_{\phi_t(x_0)}(\phi_t(x_0),\phi_{g(t)-g(0)}(x_1))=D_{x_0}(x_0,x_1) e^{-t}<\delta e^{-t}
\]
for all $t\geq 0$ and for $t=g(0)>\tau$ we have
\[
D_{x_1}(x_1,x'_2)=D_{x_0}(x_0,x_1) e^{-t}
< \frac{\eta}{3}.
\]
We can apply \eqref{ecuCambioSec} to obtain
\[
|D_{x_1}(x_1,x'_2)-D_{x_0}(x_1,x_2)|<\frac{\gamma}{3}
\]
where $x_2=\pi(x'_2)$; and 
\[
D_{x_0}(x_1,x_2)<\frac{\gamma}{3}+\frac{\eta}{3}.
\]
Therefore $x_2\in H_{\delta}(x_0)$ and
\[
D_{x_0}(x_0,x_2)\leq 
D_{x_0}(x_0,x_1)+D_{x_0}(x_1,x_2)<\delta+\frac{\gamma}{3}+\frac{\delta}{3}<\gamma,
\]
because $\eta<\gamma/3$. 
Since $x_2=\pi(x'_2)$ we have that $x_2=\phi_{s_2}(x'_2)$ with $|s_2|<\eps$ (recall that $\eps$ is the time of the flow box). 
The inductive definition is now clear, but we make it explicit.
Having $D_{x_0}(x_0,x_n)<\gamma$, there is an increasing continuous function $g_n\colon\R\to\R$ such that $g_n(0)=0$, 
$\phi_{g_n(t)}(x_n)\in H_\epsilon(\phi_t(x_0))$ and
\begin{equation}
\label{ecuAcompaña}
D_{\phi_t(x_0)}(\phi_t(x_0),\phi_{g_n(t)}(x_n))\leq\gamma e^{-t}\text{ for all }t\geq 0.
\end{equation}
For $t=g(0)$ we define $x'_{n+1}=\phi_{g_n(g(0))}(x_n)$ and 
$x_{n+1}=\pi(x'_{n+1})$.
We have 
$$D_{x_1}(x_1,x'_{n+1})<\frac{\gamma}{3},$$
\[
|D_{x_1}(x_1,x'_{n+1})-D_{x_0}(x_1,x_{n+1})|<\frac {\gamma}{3}
\]
which implies
\[
D_{x_0}(x_1,x_{n+1})<\frac{\gamma}{3} + D_{x_1}(x_1,x'_{n+1})<\frac {2\gamma}{3}.
\]
Thus
\[
D_{x_0}(x_0,x_{n+1})\leq D_{x_0}(x_0,x_1)+D_{x_0}(x_1,x_{n+1})<\gamma.
\]

Let us show that $g_n(g(0))$ is bounded. By definition, the local cross sections are closed subsets of the space. 
Consider the set
\[
\begin{array}{rl}
A=\{y\in H_\delta(x_0): & \text{exists }j\colon\R\to\R\text{ continuous }, j(0)=0\text{ and }\\ 
& \phi_{j(t)}(y)\in H_\delta(\phi_t(x_0))\text{ for all }t\in [0,g(0)]\}.
\end{array}
\]
By Lemma \ref{lemConvUnifParam} we have that $A$ is closed.
	Also, Lemma \ref{lemConvUnifParam} implies that for all $y\in A$ and $\epsilon>0$ there is $\delta>0$ such that if $z\in A$ and $\dist(y,z)<\delta$ then $$|j_u(g(0))-j_v(g(0))|<\epsilon,$$ where $j_y,j_z$ are the parameterizations given in the definition of the set $A$ for $y,z$ respectively. As $A$ is compact, taking a finite cover of such $\delta$-balls we see that the values $j(g(0))$ are bounded.
From \eqref{ecuAcompaña} we see that each $x_n\in A$ and we conclude that $g_n(g(0))$ is bounded.
Notice that $g_n(g(0))$ is the time from $x_n$ to $x'_{n+1}$. Thus, $\xi>0$ the time from 
$x_n$ to $x_{n+1}$ is less than $g_n(g(0))+\xi$; \textit{i.e.} the time from 
$x_n$ to $x_{n+1}$ is bounded.
This implies that $\omega(x_0)$ contains a periodic point in $H_{\de}(x_0)$.

Since the flow is expansive, there is a finite number of periodic orbits with period less than any fixed constant. 
This implies that a spiral point may have at most one periodic orbit in its $\omega$-limit set. This proves that $\omega(x_0)$ is a periodic orbit.
\end{proof}

\section{Proof of Theorem \ref{minimalsets}}\label{SecMinSets}

In this section, we prove Theorem \ref{minimalsets}. In contrast with the proof of Theorem \ref{minexp}, where the fields of cross-sections played a central role, here we go in the opposite direction and use cross-sections as a way of discretizing time in regular flows. This approach follows the lines of the work \cite{KS}. The reason we need this discretization is that when one intersects one-dimensional invariant sets with a cross-section, the resultant set is zero-dimensional. In this way, we dispose of several useful tools from topological dimension theory to achieve our goals.
To begin with, we need to set up  notation and collect some tools. The first one is the aforementioned result regarding the dimension of subsets of cross-sections.

\begin{proposition}\label{dimCS}
    Let $\Lambda$ be a compact, invariant, and one-dimensional subset of $X$. If $x\in \Lambda$ and $H(x)$ is a cross-section through $x$, then $\Lambda\cap H(x)$ is a zero-dimensional set.  
\end{proposition}
\begin{proof}
    The proof goes by contradiction. Suppose $\Lambda\cap H(x)$ has  positive dimension, then there is a compact set $$K\subset (\Lambda\cap H(x))$$ with dimension equals one. Since $H(x)$ is a cross-section, then any point in  $K$ belongs to a unique orbit arc contained in $\phi_{[-\tau,\tau]}(H(x))$. So,   the set $\phi_{[-\tau,\tau]}(K)$ is homeomorphic to the product $[-\eps,\eps]\times K$,  and it is contained in $\Lambda$, by invariance. By \cite{Hur}, the topological dimension of $\Lambda$ is at least two, a contradiction.
\end{proof}

Next,  let us fix the notation. Since the flows under consideration are regular,  by \cite{ArFCS}, we can  find $\tau_0,\delta_0>0$ such that for any $0<\tau\leq\tau_0$ and any $0<\delta\leq \delta_0$ there is a field  of cross-sections $$H:X\to \mathcal{K}(X),$$
of time $\tau$ and diameter $\delta$. 
Following \cite{KS}*{Lemma 2.4}, there is a finite collection of points $\{x_1,...,x_n\}$ and a finite collection of cross-sections $T(x_i)\subset H(x_i)$
satisfying

 $$X=\bigcup_{i=1}^n \phi_{[-\tau,0]}(H(x_i))=\bigcup_{i=1}^n \phi_{[0,\tau]}(H(x_i))=\bigcup_{i=1}^n \phi_{[-\tau,0]}(T(x_i))=\bigcup_{i=1}^n \phi_{[0,\tau]}(T(x_i)).$$ 
 For each $i=1,..,n$, denote $T_i=T(x_i)$, $H_i=H(x_i)$. In additon, denote  $$H=\bigcup_{i=1}^n H_i \textrm{ and } T=\bigcup_{i=1}^n T_i.$$ 
 
Now, we introduce a discretized version of $\phi$, restricted to $H$ and $T$. For this purpose, observe that for every $x\in T$, there are $1\leq j\leq n$ and $t_x\in(0,\tau]$ such that $$\phi_{t_x}(x)\in T \textrm{ and } \phi_{(0,t_x)}(x)\cap T=\emptyset.$$ 
 Therefore, we can define a map $\Phi:T\to T$ by setting $\Phi(x)=\phi_{t_x}(x)$. Observe that $\Phi$ is invertible and $\Phi^{-1}(x)=\phi_{s_x}(x)$, where $-\tau<s_x<0$ is the first negative time satisfying $\phi_{s_x}(x)\in T$.    
 By construction,  for every $x\in T$ there is an increasing sequence of times $(t_k)_{k\in \Z}$  and a a sequence cross-sections $(T_k)_{k\in \Z}$ such that $$\Phi^k(x)=\phi_{t_k}(x)\in T_k\subset H_k.$$

Notice that, by the choice of the cross-sections $T$ and $H$, $\Phi$ can be extended to $H$.  Moreover, after possibly reducing the diameter of $T$, by the continuity of $\phi$  we can find $\eta_0>0$ such that If $x\in T$ and $y\in H$ satisfy   $d(y,x)\leq\eta_0$, then  $\Phi(y)\in Int(H_1)$, where the interior is being considered in the induced topology of $H$. In particular, this ensures a certain continuity to $\Phi|_T$. Inductively, if $d(\Phi^k(x),\Phi^k(y)\leq \eta_0$, we can define $\Phi^{k+1}(y)\in Int(H_{k+1})$, for every $k\in \Z$.
Any pair of families of cross-sections $T$ and $H$ satisfying the previous properties is called a pair of \textit{adapted cross-sections}.

With this notation, we can proceed to define the concept of local stable and local unstable sets for $\Phi$, with respect to a given pair of adapted cross-sections.  

\begin{definition}
    Let $T$ and $H$ be a pair of adapted cross-sections, $\eta_0$ as above, $0\leq\eta\leq\eta_0$, and $x\in T$, we define the local stable and local unstable sets of $x$, respectively, by

    $$W^s_{\eta}(x)=\{y;d(\Phi^k(x),\Phi^k(y))\leq\eta, \forall k\geq 0\}.$$
    $$W^u_{\eta}(x)=\{y;d(\Phi^k(x),\Phi^k(y))\leq\eta, \forall k\geq 0\}.$$
    
\end{definition}
  The next result gives us a characterization of expansiveness through pairs of adapted cross-sections.

\begin{theorem}[\cite{KS}]
    A flow $\phi$ is expansive if, and only if, for every pair of families of adapted cross-sections inducing $\Phi$ there is $0<\eta\leq \eta_0$ such that $W_{\eta}^s(x)\cap W_{\eta}^s(x)=\{x\}$.   
\end{theorem}

The constant $\eta$ above is called an expansiveness constant of $\Phi$. The next two lemmas are borrowed from \cite{ACP}. 
We denote by $\SS_{\eta}(x)$ and $\SU_{\eta}(x)$, the connected components of $W^s_{\eta}(x)$ and $W^u_{\eta}(x)$ containing $x$, respectively.

\begin{lemma}[\cite{ACP}]\label{bigstable}
     Let $H$ and $T$ be adapted families of cross-sections, $\eta>0$ an expansiveness constant for $\Phi$, and suppose $\dim(X)>1$. Then there is $x\in T$ such that either  $diam(\SS_{\eta}(x))>c>0$ or $diam(\SU_{\eta}(x))>c>0$.
\end{lemma}

The previous lemma tells us that if $\dim(X)>1$, then there is at least a non-trivial piece of connected local stable or unstable. On the other hand, the next lemma says that the discretized flow  will expand any non-trivial continuum  with big diameter, in controlled time.

\begin{lemma}[\cite{ACP}]\label{futuretimes}
    Let $H$ and $T$ be adapted families of cross-sections and $\eta$ an expansiveness constant for $\Phi$. There is $0<\kappa\leq \frac{\eta}{2}$ such that for every $\xi>0$ there is $N>0$ such that 
    if $C\subset T$ is a continuum with diameter $diam(C)\geq \xi$  then either $$diam(\Phi^n(C))\geq \kappa \textrm{ or } diam(\Phi^{-n}(C))\geq \kappa,$$ for every $n\geq N$.

\end{lemma}

The next tool needed is the following result, which is contained in the proof of  \cite[theorem 2.6]{Kato95}, and we extracted it in form of Lemma for our purposes.

\begin{lemma}\label{cover}
    Let $X$ be a compact metric space and $K\subset X$ be a zero-dimensional compact subset of $X$. For every $\alpha>0$ there is an open cover $\{U_1,...,U_n\}$ such that:
    \begin{enumerate}
        \item $diam(U_i)\leq \alpha$, for every $i=1,...n$
        \item $\overline{U_i}\cap \overline{U_j}=\emptyset$, whenever $i\neq j$.
        \item There is $\lambda>0$ such that  every continuum $A\subset X$ with $diam(A)>3\alpha$, there is a subcontinuum $B\subset A$ such that $diam(B)\geq \lambda$ and $B\cap \overline{U_i}=\emptyset$, for every $i=1,...,n$.
    \end{enumerate}
\end{lemma}

The next result is the last ingredient needed for the proof of Theorem \ref{minimalsets}.

\begin{theorem}\label{one-dim}
    Let $\phi$ be an expansive flow on $X$ with $\dim(X)>1$. If $K\subset X$ is one-dimensional, compact, and invariant, then there is a minimal subset disjoint from $K$.
\end{theorem}

\begin{proof}
Let $T$ and $S$ be a pair of adapted cross-sections for $\phi$. Let $\Phi$ be the discretized flow of $\phi$ and let $\eta$ be a constant of expansiveness for $\Phi$.   Since $\dim(X)>1$, there is $p\in T$ such that either $$\diam(\SS_{\eta}(p))>c \textrm{ or } \diam(\SU_{\eta}(p))>c.$$  Let us assume $\diam(\SU_{\eta}(p))\geq c$, otherwise we just reproduce the next argument to the map  $\Phi^{-1}$. Denote $K_0=T\cap K$ and observe that $K_0\neq \emptyset$ due to the invariance of $K$. Moreover, $K_0$ is zero-dimensional because  of Proposition \ref{dimCS}. 

Fix $$0<\kappa\leq\min\left\{\frac{\eta}{2},c\right\} \textrm{ and fix } 0<\alpha\leq\frac{\kappa}{3}.$$ Let $\{U_1,...,U_i\}$ be an open cover of $K_0$ in $T$  
 be given by Lemma \ref{cover}. Fix  $\lambda>0$ as in Lemma \ref{cover} and choose $0<\xi\leq \lambda$.

Since $\diam(\SU_{\eta}(p))\geq c$, there is a continuum $C\subset \SU_{\eps}(p)$ with $$\diam(C)\geq \lambda \geq \xi \textrm{ and } C\cap \overline{U_i}=\emptyset,$$ for $i=1,...,n$. Hence, the fact that $C\subset \SU_{\eps}(p)$ implies, by Lemma \ref{futuretimes}, the existence of  $N>0$ such that 
$$\diam(\Phi^k(C))\geq \kappa,$$ for every $k\geq N$. 
So, we can find a continuum $B_1\subset \Phi^{N}(C)$, such that  $$\diam(B_1)\geq \xi,$$ and $B_1\cap \overline{U_{i}}=\emptyset$, for every $i=1,...,n$. 

Inductively, for every  $j\geq 1$, there is a continuum $B_J\subset \Phi^{jN}(C)$, such that $$\diam(B_j)\geq\xi,$$ and $B_1\cap \overline{U_{i}}=\emptyset$ for every $i=1,...,n$. Consequently, there is $x\in C$, such that $$d(\Phi^{jN}(x), U_{i})\geq\frac{\lambda}{2},$$ for every $j\geq 0$ and every $i=1,...,n$. Now, by the continuity of $\Phi$, it is easy to see that there is $\gamma>0$ such that $d(\Phi^i(x),K)\geq \gamma$, for every $i\geq 0$. In particular, $d(\overline{orb^+(x)},K)>0$ and therefore $\omega(x)$ contains a minimal set disjoint from $K$. 
\end{proof}

\begin{proof}[Proof of Theorem \ref{minimalsets}]
Suppose $X$ has at least dimension two and let $\phi$ be an expansive flow on $X$ without periodic orbits. By Theorem \ref{minexp}, $\phi$ cannot be minimal. So  $X$ contains a minimal set, which is one-dimensional, again by Theorem \ref{minexp}. To prove that we actually have infinitely minimal subsets, suppose $X$ contains only finitely many subsets and denote those sets by $K_1,..., K_n$. Since $K_i$, is one dimensional for every $i=1,...,n$, then $$K=\bigcup_{i=1}^n K_1 $$
is also one-dimensional. Thus, Theorem \ref{one-dim} implies the existence of a minimal set $K'$ disjoint from $K$, a contradiction. This concludes the proof.
\end{proof}

\section{Proof of Theorem \ref{charac}}
\label{secCharac}
This section is devoted to providing the reader with a proof for Theorem \ref{charac} and exploring some of its consequences. Remember that our goal in Theorem \ref{charac} is to obtain a characterization of expansiveness through the assumption of expansiveness on $X'=X\setminus Orb(x)$, where $x$ is any point of $X$. Observe that the set is not necessarily compact, and this represents the biggest difficulty here. To deal with this issue, 
we begin by choosing a field of cross-sections $H$ with diameter $\eps_0>0$ and time $\tau>0$. Next, for each $0<2\eps\leq\eps_0$ define the \textit{dynamical ball} of $x$ with respect to  $H$ as
\[
\Gamma_\epsilon(x)=\{y\in H_\epsilon(x):\exists h\in Rep \text{ s.t. } 
\phi_{h(t)}(y)\in H_\epsilon(\phi_t(x)), \forall t\in\R\}.
\]
The next result is based on the proof of \cite{ACCV}*{Proposition 2.3} and is an important ingredient for our proof.

\begin{lemma}
\label{lemNoExpOrb}
If the restriction of $\phi$ to the orbit of $x$ is not expansive then 
$\Gamma_\epsilon(x)$ is uncountable for all $\epsilon>0$.
\end{lemma}

\begin{proof}
Given $\epsilon>0$
define the subsets $A_n\subset H_\epsilon(x)$ as 
$A_1=\Gamma_{\epsilon/2}(x)\cap \orb(x)$ and 
\[
\begin{array}{rl}
A_{n+1}&=\{y\in H_\epsilon(x)\cap\orb(x):\exists z\in A_n \text{ and } g,h\in Rep \text{ s.t. }\\
& \phi_{g(t)}(y),\phi_{h(t)}(z)\in H_\epsilon(\phi_t(x))\text{ and }\dist(\phi_{g(t)}(y),\phi_{h(t)}(z))\leq\epsilon/2^{n+1}, \forall t\in\R\}.
\end{array}
\]
Since the flow is not expansive on the orbit of $x$, we have that no $y\in\orb(x)$ is isolated in $\Gamma_r(y)$, for all $r>0$.
This implies that each point in $A_n$ is not isolated in $A_{n+1}$.
Therefore, the set $B=\cup_{n\geq 1} A_n$ has no isolated point and its closure is contained in $\Gamma_\epsilon(x)$.
As the closure of $B$ is uncountable, the proof ends.
\end{proof}

Now we are in a position to prove Theorem \ref{charac}.

\begin{proof}[Proof of Theorem \ref{charac}]

Suppose that $\phi$ is a continuous regular flow on $X$ and fix $x_1\in X$. Denote $$X'=  X\setminus \orb(x_1).$$ 
If $\phi$ is expansive, then $\phi$ is obviously expansive on $X'$. Hence we just need to prove the converse.

Suppose the restriction of $\phi|_{X'}$ is expansive. 
We will show that $\phi$ is expansive on $X$.
Fix $\delta_0,\tau_0>0$  such that for every $0<\delta\leq \delta_0$ and $0<\eps\leq \tau_0$, there is a field of cross-sections  with diameter $\delta$ and time $\eps$. 
Fix $0<\eps\leq\tau_0$ and let $0<2\delta_1\leq \delta_0$ be given by the expansiveness of $\phi$ on $X'$.  Let $H:X\to \mathcal{K}(X)$ be field of cross sections for $\phi$ with diameter $2\delta_1$ and time $\eps$.

To conclude that $\phi$ is expansive, we must find an adequate constant $\delta>0$ such that if $x,y\in X$, $h\in Rep$ and $$d(\phi_t(x),\phi_{h(t)}(y))\leq \delta,$$ for every $t\in \R$, then $y\in \phi_{[-\eps,\eps]}(x)$. This $\delta$ will then be taken as the minimum between three constants $\delta_1,\delta_2,\delta_3$, which will be defined depending on the points $x,y$ are located at: \\

\textbf{Case 1:}  Suppose $x,y\in X'$. In this case, the expansiveness of  $\phi$ in $X'$ implies that if $d(\phi_t(x),\phi_{h(t)}(y))\leq \delta_1$, for every $t\in \R$, then $y\in \phi_{[-\eps,\eps]}(x)$.
\\

\textbf{Case 2:}  Now, suppose $x\in X'$ and $y\in Orb(x_1)$. To obtain a proof in this case,  we first need the following claim:
\begin{claim}
    There is at most on point $z\in X'\cap H(x_1)$, so that there is $h\in Rep$ satisfying $d(\phi_t(x_1),\phi_{h(t)}(z))\leq \frac{\delta_1}{2}$, for every $t\in \R$.
\end{claim}
\begin{proof}
 Otherwise, there should be $z',z\in H(x_1)\cap X'$, $z'\neq z$ and $h,g\in Rep$ such that 
 $$d(\phi_t(x_1),\phi_{h(t)}(z))\leq \frac{\delta_1}{2} \textrm{ and } d(\phi_t(x_1),\phi_{g(t)}(z'))\leq \frac{\delta_1}{2},$$
 for every $t\in R$.  In this case, we obtain
$$d(\phi_{h(t)}(x),\phi_{g(t)}(y))\leq \delta_1,$$
for every $t\in \R$. Therefore, there is $h'\in Rep$ so that $$d(\phi_{t}(x),\phi_{h'(t)}(y))\leq \delta_1,$$
for every $t\in \R$. By assumption $\phi$ is expansive on $X'$ and hence $z'\in \phi_{[-\eps,\eps]}(z)$, but this impossible because $H(x_1)$ is a cross-section of time $\eps$ and $z\neq z'$. 
 \end{proof}

Let $z$ be given by the claim, and  consider $$0<\delta_2\leq\frac{1}{2}\min\left\{d(\phi_{[-\eps,\eps]}(z), x_1),\delta_1\right\}.$$ 
Now, suppose  $$d(\phi_{t}(x),\phi_{h'(t)}(y))\leq \delta_2,$$
for every $t\in \R$. Since $y\in Orb(x_1)$, there is $t_0\in \R$ so that $\phi_{h(t_0)}(y)=x_1$. By writing $x_0=\phi_{t_0}(x)$ we can obtain $g\in Rep$ so that $$d(\phi_{t}(x_0),\phi_{g(t)}(x_1))\leq \delta_2,$$
for every $t\in \R$. But the claim implies $x_0\in \phi_{[-\eps,\eps]}(H(x_1))$ and by the choice of $\delta_2$ we must have $x_0\in \phi_{[-\eps,\eps]}(x)$, contradicting the assumption $x\in X'$.\\

\textbf{Case 3:} Suppose  $x,y\in Orb(x_1)$. First, we claim that $H(x_1)\cap Orb(x_1)$ is countable. Indeed, since $H(x_1)$ is a cross-section of time $\eps$, there if $t\neq s$ are such that $$\phi_{t}(x_1),\phi_{s}(x_1)\in H(x_1),$$ then $|t-s|>\eps$. Consequently, every point in the set $$A=\{t\in \R;\phi_t(x_1)\in H(x_1) \},$$
 isolated and hence $H(x_1)\cap Orb(x_1)$ is countable. 
 
 Now, arguing by contradiction, we can apply Lemma \ref{lemNoExpOrb} to obtain that $\Gamma_r(x_1)$ is uncountable for all $r>0$. Thus, taking two points in $\Gamma_r(x_1)\setminus\orb(x_1)$ we contradict the expansivity in $X'$. Consequently, there is a constant of expansivity $\delta_3>0$ for $\phi$ in $Orb(x_1)$. Finally, by choosing $0<\delta\leq \min\{\delta_1,\delta_2,\delta_3\}$, we conclude that $\phi$ is expansive.

\end{proof}

Suppose that $\phi$ and $\psi$ are flows on the same compact space $X$, $\phi$ is regular and $\psi$ has just one singular point $p\in X$. We say that $\psi$ \textit{adds a singular point} to $\phi$, or $\phi$ \textit{removes a singular point} from $\psi$, if every orbit of $\psi$ is contained in an orbit of $\phi$ with the same flow direction. 
In this case, $\phi_\R(x)=\psi_\R(x)$ for all $x\notin \phi_\R(p)$.

For the next result we need \cite{ArExpFSurf}*{Theorem 1.3} which states that 
a flow $\psi$ on the compact space $X$ is $k^*$-expansive if and only if 
for all $\beta>0$ there is $\delta>0$ such that 
if $x,y\in X$, $h\in Rep$ and $d(\phi_{h(t)}(x),\phi_t(y))<\delta$ for all $t\in\R$ then 
$x,y$ are in an orbit segment of diameter less than $\beta$.

\begin{corollary}
Suppose that $\phi$ is a regular flow of the compact space $X$ and $\psi$ adds a singular point $p\in X$ to $\phi$. 
Then, $\phi$ is expansive if and only if $\psi$ is $k^*$-expansive.
\end{corollary}

\begin{proof}
The proof of the direct part was given in \cite{ArExpFSurf}*{Proposition 6}. 
To prove the converse suppose that $\psi$ is $k^*$-expansive. 
We will show that the regular flow $\phi$ is expansive by applying Theorem \ref{charac}. 
Thus we need to prove that $\phi$ is expansive in $X\setminus \orb_{\phi}(p)$.
Given $\epsilon>0$, as $\phi$ is regular we can take $\beta>0$ such that 
if 
\begin{equation}
\label{ecuBeta}
\diam(\phi_{[0,s]}(z))<\beta\text{ then }s<\epsilon.
\end{equation}
The $k^*$-expansivity of $\psi$ and \cite{ArExpFSurf}*{Theorem 1.3} gives us the corresponding expansivity constant $\delta>0$.

Suppose that $x,y\notin\orb_\phi(p)$ and 
$h\in Rep$ such that 
$d(\phi_t(x),\phi_{h(t)}(y))\leq \delta,$ for every $t\in \R$.
This implies that there is $g\in Rep$ such that 
$d(\psi_t(x),\psi_{g(t)}(y))\leq \delta,$ for every $t\in \R$.
Thus, $x$ and $y$ are in an orbit segment of diameter less than $\beta$.
From \eqref{ecuBeta} we have that $y=\phi_s(x)$ for some $|s|<\epsilon$.
This proves that $\phi$ is expansive in the complement of the orbit of $p$ and from Theorem \ref{charac} we conclude that $\phi$ is expansive.
\end{proof}

\vspace{0.1in}

\section*{Acknowledgements}

The authors thank the valuable comments from the anonymous referee, which helped to improve the presentation of this work.   The second author is grateful to the Southern University of Science and Technology Department of Mathematics, where part of this work was developed. The first author is partially supported by Agencia Nacional de Investigaci\'on e Inovaci\'on (ANII) and Programa de Desarrollo de las Ciencias B\'asicas (PEDECIBA).
\begin{table}[h]
\begin{tabularx}{\linewidth}{p{1.5cm}  X}
\includegraphics [width=1.4cm]{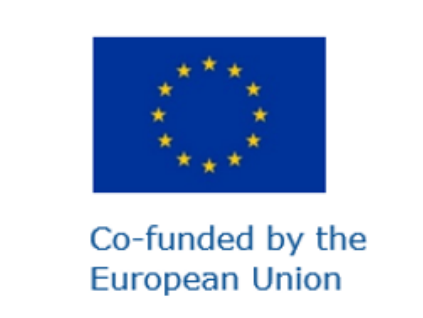} &
\vspace{-1.3cm}
This research is part of a project that has received funding from
the European Union's European Research Council Marie Sklodowska-Curie Project No. 101151716 -- TMSHADS -- HORIZON--MSCA--2023--PF--01.\\
\end{tabularx}
\end{table}


\begin{bibdiv}
\begin{biblist}

\bib{AH}{book}{
   author={N. Aoki},
   author={K. Hiraide},
   title={Topological theory of dynamical systems},
   series={North-Holland Mathematical Library},
   volume={52},
   note={Recent advances},
   publisher={North-Holland Publishing Co., Amsterdam},
   date={1994},
   pages={viii+416},
   isbn={0-444-89917-0},
   review={\MR{1289410}},
}
\bib{ACP}{article}{
   author={A. Arbieto},
   author={W. Cordeiro},
   author={M.J. Pacifico},
   title={Continuum-wise expansivity and entropy for flows},
   journal={Ergodic Theory Dynam. Systems},
   volume={39},
   date={2019},
   number={5},
   pages={1190--1210}
}

\bib{ArExpFSurf}{article}{
author={A. Artigue},
title={Expansive flows of surfaces},
journal={Discrete and Continuous Dynamical Systems},
year={2013},
volume={33},
pages={505-525},
doi={10.3934/dcds.2013.33.505}}

\bib{ArMin}{article}{
	author = {A. Artigue},
	title = {Minimal expansive systems and spiral points},
	journal={Topology and its applications},
	volume={194},
	year={2015},
	pages={166--170}}%

\bib{ArFCS}{article}{
author={A. Artigue}, 
title={Discrete and continuous topological dynamics: Fields of cross sections and expansive flows}, 
journal={Discrete and Continuous Dynamical Systems}, 
volume={36}, 
year={2016}, 
pages={5911-5927}}

\bib{ArCatenary}{article}{
author={A. Artigue},
title={Isolated sets, catenary Lyapunov functions and expansive systems},
journal={Topological Methods in Nonlinear Analysis},
year={2017}}

\bib{ACCV}{article}{
author={A. Artigue},
author={B. Carvalho},
author={W. Cordeiro},
author={J. Vieitez},
title={Countably and entropy expansive homeomorphisms with the shadowing property},
journal={Proceedings of the AMS},
year={2022},
pages={3369-3378}}

\bib{BW}{article}{
   author={R. Bowen},
   author={P. Walters},
   title={Expansive one-parameter flows},
   journal={J. Differential Equations},
   volume={12},
   date={1972},
   pages={180--193},
   issn={0022-0396},
   review={\MR{0341451}},
   doi={10.1016/0022-0396(72)90013-7},
}



\bib{CP}{article}{
   author={Choi, Sung Kyu},
   author={Park, Jong Suh},
   title={Hyperbolic flows are topologically stable},
   journal={Bull. Austral. Math. Soc.},
   volume={43},
   date={1991},
   number={2},
   pages={225--232},
   issn={0004-9727},
   review={\MR{1097061}},
   doi={10.1017/S0004972700028987},
}



\bib{Hur}{article}{
   author={W. Hurewicz},
   title={Sur la dimension des produits cartesiens},
   language={French},
   journal={Ann. of Math. (2)},
   volume={36},
   date={1935},
   number={1},
   pages={194--197},
   issn={0003-486X},
   review={\MR{1503218}},
   doi={10.2307/1968674},
}


\bib{Kato93}{article}{
	author={H. Kato},
	title={Continuum-wise expansive homeomorphisms},
	journal={Canad. J. Math.},
	volume={45},
	number={3},
	year={1993},
	pages={576--598}}

\bib{Kato93b}{article}{
author={H. Kato},
title={Concerning continuum-wise fully expansive homeomorphisms of continua},
journal={Topology and its Applications},
volume={53},
year={1993},
pages={239--258}}

\bib{Kato95}{article}{
author={H. Kato}, 
title={Minimal sets and chaos in the sense of Devaney on continuum-wise expansive homeomorphisms},
journal={Lecture Notes in Pure and Appl. Math.}, 
volume={170}, 
year={1995}, 
pages={265--274}}
	
\bib{Kato96}{article}{
author={H. Kato},
title={Chaos of continuum-wise expansive homeomorphisms and dynamical properties of sensitive maps of graphs},
journal={Pacific Journal of Mathematics},
volume={175}, 
number={1}, 
year={1996}}
	
\bib{Kato03}{article}{
    author={H. Kato},
    title={Infinite Minimal Sets Of Continuum-Wise Expansive Homeomorphisms Of 1-Dimensional Compacta},
    journal={Topology and its Applications},
    volume={130},
    year={2003},
    pages={57--64}}

\bib{KS}{article}{
author = {H.B. Keynes},
author={M. Sears},
title = {Real-expansive flows and topological dimension},
journal = {Ergodic Theory and Dynamical Systems},
volume = {1},
year = {1981},
pages = {179--195}}


\bib{Ma}{article}{
	author={R. Mañé},
	title={Expansive homeomorphisms and topological dimension},
	journal={Trans. of the AMS}, 
	volume={252}, 
	pages={313--319}, 
	year={1979}}




\bib{MoSi}{book}{
   author={C.A. Morales},
   author={V.F. Sirvent},
   title={Expansive measures},
   series={Publica\c c\~oes Matem\'aticas do IMPA. [IMPA Mathematical
   Publications]},
   note={29$\sp {\rm o}$ Col\'oquio Brasileiro de Matem\'atica. [29th
   Brazilian Mathematics Colloquium]},
   publisher={Instituto Nacional de Matem\'atica Pura e Aplicada (IMPA), Rio
   de Janeiro},
   date={2013},
   pages={viii+89},
   isbn={978-85-244-0360-6},
   review={\MR{3100422}},
}


\bib{Nadler}{book}{
	author={S. Nadler Jr.}, 
	title={Continuum Theory}, 
	series={Pure and Applied Mathematics}, 
	volume={158}, 
	publisher={Marcel Dekker, New York},
	year={1992}} 



\bib{Pa}{article}{
   author={M. Paternain},
   title={Expansive geodesic flows on surfaces},
   journal={Ergodic Theory Dynam. Systems},
   volume={13},
   date={1993},
   number={1},
   pages={153--165},
   issn={0143-3857},
   review={\MR{1213085}},
   doi={10.1017/S0143385700007264},
}


\bib{U}{article}{
   author={W.R. Utz},
   title={Unstable homeomorphisms},
   journal={Proc. Amer. Math. Soc.},
   volume={1},
   date={1950},
   pages={769--774},
   issn={0002-9939},
   review={\MR{0038022}},
   doi={10.2307/2031982},
}


\bib{Walters}{book}{
author={P. Walters},
title={An Introduction to Ergodic Theory},
publisher={Springer-Verlag New York, Inc.},
year={1982}}


\bib{Wi}{article}{
   author={R.K. Williams},
   title={On expansive homeomorphisms},
   journal={Amer. Math. Monthly},
   volume={76},
   date={1969},
   pages={176--178},
   issn={0002-9890},
   review={\MR{0238278}},
   doi={10.2307/2317269},
}

\end{biblist}
\end{bibdiv}
\end{document}